\documentclass{amsart}
\usepackage[utf8]{inputenc}
\usepackage{amsfonts}
\usepackage{amsmath}
\usepackage{amssymb}
\usepackage{amsthm}
\usepackage{xcolor}
\usepackage{bm}
\usepackage{mathtools}
\usepackage{enumitem}
\usepackage{verbatim}
\usepackage{pinlabel}
\usepackage{caption}
\usepackage{subcaption}
\usepackage{tikz}
\usetikzlibrary{matrix,arrows,decorations.pathmorphing}
\usepackage{graphicx}
\usepackage{hyperref}
\usepackage{tikz-cd}
\usepackage{biblatex}
\usepackage{float}
\usepackage{geometry}
\usepackage{changepage}

\usepackage{xcolor}
\definecolor{RRR}{RGB}{198,0,31}
\definecolor{PPP}{RGB}{255,0,170}
\definecolor{BBB}{RGB}{0,0,221}
\definecolor{GGG}{RGB}{0,120,0}

\definecolor{DB}{RGB}{13,0,155}
\definecolor{DG}{RGB}{13,137,0}

\newtheorem{theorem}{Theorem}[section]

\newtheorem{lemma}[theorem]{Lemma}

\newtheorem{proposition}[theorem]{Proposition}

\newtheoremstyle{claim}% name
  {\topsep}% space above
  {\topsep}% space below
  {}% body font
  {}% indent amount
  {\itshape}% theorem head font
  {.}% punctuation after theorem head
  {.5em}% space after theorem head
  {\thmname{#1}\thmnumber{ #2}\thmnote{ (#3)}}% theorem head spec
\theoremstyle{claim}

\DeclareMathOperator{\interior}{int}
\DeclareMathOperator{\id}{id}
\DeclareMathOperator{\Homeo}{Homeo}
\DeclareMathOperator{\Diffeo}{Diff}
\DeclareMathOperator{\MCG}{MCG}

\DeclareMathOperator{\Aut}{Aut}
\DeclareMathOperator{\link}{link}

\newcommand{\fine}{{\mathcal{C}^\dagger}}

\newcommand{\surf}{S_{g}^n}
\newcommand{\weirdarc}{\mathcal{A}_{\partial,\geq2\times}}

\newcommand{\pit}[1]{\medskip\noindent\textit{#1}\textit{.}}

\newcommand{\adj}{{-}}

\newcommand{\p}[1]{\medskip\noindent\textbf{#1}\textbf{.}}

\title{Automorphisms of fine curve graphs of planar surfaces}
\author{Roberta Shapiro, Rohan Wadhwa, Arthur Wang, and Yuchong Zhang}
\begin{document}

\begin{abstract}
   The fine curve graph of a surface is the graph whose vertices are simple closed essential curves in the surface and whose edges connect disjoint curves. In this paper, we prove that the automorphism group of the fine curve graph of a surface is naturally isomorphic to the homeomorphism group of the surface for boundaryless planar surfaces with at least 7 punctures.
   \end{abstract}
\maketitle

\section{Introduction}

Let $S=\surf$ be a boundaryless orientable surface of genus $g$ with $n$ punctures. Unless otherwise stated, we will assume that $g=0$ and $n\geq 7.$ The fine curve graph of $S,$ denoted $\fine(S),$ is the graph whose vertices are essential simple closed curves and whose edges connect pairs of curves that are disjoint. We prove the following theorem.

\begin{theorem}\label{maintheorem}
    Let $S=S_0^n$ be a boundaryless orientable surface with $n\geq 7.$ Then the natural map
    \[\Phi:\Homeo(S)\to \Aut \fine(S)\]
    is an isomorphism.
\end{theorem}

Theorem~\ref{maintheorem} is an extension of the theorem of Long--Margalit--Pham--Verberne--Yao, which states the analogous result for compact and boundaryless surfaces of genus at least 2 \cite{LMPVY}. More generally, our theorem is another``fine" analogue of the theorems of Ivanov \cite{Ivanov}, Korkmaz \cite{Korkmaz}, and Luo \cite{luo} that the extended mapping class group of a surface is naturally isomorphic to the automorphisms of its curve graph. (The \emph{curve graph} of a surface is a graph whose vertices are isotopy classes of essential simple closed curves and whose edges connect classes that admit disjoint representatives.)

%\p{(Fine) curve graphs, past and present} The \emph{curve graph} of a surface is a graph whose vertices are isotopy classes of essential simple closed curves and whose edges connect classes that admit disjoint representatives. %Curve graphs and their variants are commonly used to study the \emph{mapping class groups} of surfaces, which are groups of homeomorphisms modulo isotopy. In fact, 
%Ivanov \cite{Ivanov}, Korkmaz, \cite{Korkmaz}, and Luo \cite{luo} prove that the extended mapping class group of a surface is naturally isomorphic to the automorphisms of the surface's curve graph.

Fine curve graphs were first introduced by Bowden--Hensel--Webb for compact, boundaryless surfaces in the case that vertices represent smooth curves; Bowden--Hensel--Webb use fine curve graphs to show that $\Diffeo_0(S_g^0)$ admits unbounded quasimorphisms, and as such is not uniformly perfect \cite{Bowden_Hensel_Webb_2021}. Since then, Long--Margalit--Pham--Verberne--Yao have proven that the homeomorphism group of an orientable, compact, boundaryless surface of genus at least 2 is naturally isomorphic to the automorphism group of the fine curve graph of said surface. A similar result was shown for non-orientable surfaces of genus at least $4$ by Kimura--Kuno \cite{kk}.

Even more recently, people have explored the automorphism groups of fine curve graph variants. Fine curve graphs with edges not corresponding to disjointness were introduced in Le~Roux--Wolff \cite{LRW} and Booth--Minahan--Shapiro \cite{BMS} and were further explored by Shapiro \cite{shapiro}; all three papers study automorphism groups of such graphs. Booth also studies automorphism groups of graphs whose vertices correspond to curves with certain differentiability conditions \cite{booth1} \cite{booth2}.

\p{Proof outline} The proof of Theorem~\ref{maintheorem} is summarized by the following commutative diagram.
\begin{center}
    \begin{tikzcd}
 \Aut \fine(S) \arrow{r}{\Psi} & \Aut \mathcal{E}\fine(S) \arrow{r}{\nu^{-1}} &\Homeo(S) \arrow[bend left]{ll}{\Phi}
 \end{tikzcd}
 \end{center}
We factor through the extended fine curve graph, denoted $\mathcal{E}\fine(S),$ just as Long--Margalit--Pham--Verberne--Yao do \cite{LMPVY}, and show that automorphisms of the fine curve graph induce automorphisms of the extended fine curve graph via the map $\Psi$ above. The main idea is that the extended fine curve graph of a surface includes vertices corresponding to inessential curves, which ultimately encode points on $S.$ We use Theorem~1.2 of Long--Margalit--Pham--Verberne--Yao, paraphrased as Theorem~\ref{theorem:efcg}, stating that the natural map $\nu:\Homeo(S)\to \Aut\mathcal{E}\fine(S)$ is an isomorphism.

Leveraging this, we show that inessential curves are encoded by pairs of essential curves, called bigon pairs, and that automorphisms of the fine curve graph preserve both bigon pairs and the property of bigon pairs encoding the same inessential curves.

\p{Distinctions from the non-planar cases} The first and clearest distinction we must account for when considering surfaces with punctures is that we now have inessential curves that are homotopic to a puncture. Although we do not consider such curves as vertices in the extended fine curve graph, we have to account for them in our definitions of sides of curves, hulls, bigon pairs, and pants pairs in Section~\ref{sec:configs}. This makes our definitions different from those used by Long--Margalit--Pham--Verberne--Yao \cite{LMPVY} and Kimura--Kuno \cite{kk}. 

When it comes to curves, there is a clear distinction between planar and non-planar surfaces: planar surfaces have only separating curves (curves which, when removed, break the surface into two connected components). This leads us to have a different set of relevant configurations of curves, as we no longer have torus pairs, which are formed by curves which minimally intersect exactly once. Thus we define other curve configurations in Section~\ref{sec:configs}. Moreover, although we follow a similar outline to Long--Margalit--Pham--Verberne--Yao, the steps to each proof are different. In particular, we define an unusual arc graph in Section~\ref{sec:configs} that arises from these distinctions.

\medskip\noindent\textit{Why no boundary?} Traditionally, many people study the group of homeomorphisms that fix the boundary pointwise, $\Homeo(S,\partial S).$ Since curves cannot intersect boundary components, we end up having symmetries of $S$ that are not actually homeomorphisms. For example, imagine a plane with 10 boundary components placed in a symmetric way about the $y$-axis; a reflection over the $y$-axis would induce an automorphism of the curve graph of this surface, but would not be induced by a homeomorphism as the reflection does not fix the boundary pointwise. Thus, there would be some difficulty with defining what types of homeomorphisms we would be working with when it comes to a surface with boundary.

\pit{Why $n\geq 7$} We take $n\neq 1,2,3$ due to the lack of essential simple closed curves and $n\neq 4$ since, under the normal definition of the edges representing disjointness, the fine curve is disconnected. However, the proof enclosed only applies in its entirety to surfaces with $n\geq 7.$ This bound might not be sharp. Each intermediate result in this paper includes the tightest bound for which the proof works. 

\p{Acknowledgments} Thank you to Katherine Booth, Jacob Guynee, and Carlos P\'erez Estrada for conversations about aspects of the proof. Thank you to Alex Nolte for comments on a draft of this paper. This work was done partially through the LoG(M) program at the University of Michigan.

\section{Configurations of curves}\label{sec:configs}

%In this section, we build up the steps to proving Theorem~\ref{maintheorem} and culminate in its proof.

To begin, we recall and expand on several definitions and results of Long--Margalit--Pham--Verberne--Yao that are still true in the case of planar surfaces.

\pit{Sides} Let $A=\{a_1,\ldots,a_n\}$ be a collection of curves in $S$. Then the \emph{sides} of $A$ are the connected components of $S\setminus A$. A curve $b$ is \emph{on one side} of $A$ if $b\setminus A$ is contained in a single connected component of $S\setminus A$ that is neither a disk nor a punctured disk. Curves $b$ and $c$ are \emph{on the same side} of $A$ if they are both on one side of $A$ and $(b\cup c)\setminus A$ are in the same connected component of $S\setminus A.$ Otherwise, they are \emph{on opposite sides.} We say a curve $a$ \emph{separates} curves $b$ and $c$ if $b$ and $c$ are on opposite sides of $a.$ Finally, intersecting curves $a$ and $b$ are \emph{noncrossing} if $a$ is on one side of $b$ (and vice versa) and \emph{crossing} otherwise.

Let $G$ be a graph. We denote the set of vertices of $G$ by $V(G).$ If $x,y\in V(G),$ we will notate that they are adjacent by writing $x\adj y.$ Define the \emph{link} of the set of vertices $A$ of a graph $G$, denoted $\link(A),$ to be the subgraph induced by the set $\{v\in V(G) | v\adj a \text{ for all }a \in A\}$.

\begin{proposition}[Sides of curves. Generalization of Lemma 2.3 of \cite{LMPVY}]\label{prop:sides}
    Let $S=S_{0}^n$ be a surface with $n\geq 5$ and $A$ be a collection of curves in $S.$ Then automorphisms of $\fine(S)$ preserve the sides of $A.$ That is, for all $\varphi\in\Aut \fine(S)$ and curves $a,c\in V(\fine(S))\setminus A,$ $a$ and $c$ are on the same side of $A$ if and only if $\varphi(a)$ and $\varphi(c)$ are on the same side of $\varphi(A).$

    Moreover, if $a$ is on one side of $A,$ then $\varphi(a)$ is on one side of $\varphi(A).$
\end{proposition}

\begin{proof}
    The case that $A$ is a multicurve and $a,c\in \link(A)$ is done by Long--Margalit--Pham--Verberne--Yao \cite{LMPVY}. Their proof holds in our case as well, since $\link(A)$ is a join whose parts correspond to curves contained in the connected components of $S\setminus A.$

    We will show that $a$ and $c$ are on the same side of $A$ if and only if there exists a curve $d\in\link(A)$ that intersects both $a$ and $c.$ 

    Suppose $a$ and $c$ are on the same side of $A.$ Let $d\in S\setminus A$ be an essential curve in the same connected component of $S\setminus A$ as $a\setminus A$ and $c\setminus A.$ We can isotope $d$ to intersect both $a$ and $c.$

    Conversely, suppose $a$ and $c$ are on different sides of $A.$ Then any curve that intersects both $a$ and $c$ must also intersect some curve in $A,$ meaning no curve in $\link(A)$ intersects both $a$ and $c.$

\medskip \noindent
    Now suppose $a$ is on one side of $A.$ Then one part of the join in $\link(A)$ corresponds to the connected component of $S\setminus A$ that contains $A\setminus a.$ There exists a curve $c$ in this part that intersects $a.$ However, all curves in all other parts of the join are disjoint from $a$ since they are disjoint from $A.$ Since sides of joins are preserved by automorphisms, so is the property of being on one side of $A.$
\end{proof}

\pit{Extended fine curve graph} The \emph{extended fine curve graph} of a surface $S,$ denoted $\mathcal{E}\fine(S)$, is the graph whose vertices are simple closed curves (not necessarily essential) and whose edges connect disjoint curves. Long--Margalit--Pham--Verberne--Yao provide the first written proof that, for surfaces without boundary, $\Aut \mathcal{E}\fine(S) \cong \Homeo(S)$ \cite[Theorem 1.2]{LMPVY}. (We note that the original unpublished proof is due to Farb--Margalit \cite{fm}.) We record this theorem below.

\begin{theorem}[Theorem 1.2 of \cite{LMPVY}]\label{theorem:efcg}
    For any surface $S$ without boundary, the natural map \[\nu:\Homeo(S) \to \mathcal{E}\fine(S)\] is an isomorphism.
\end{theorem}

\subsection{Quasi-homotopic curves, homotopic curves, and pants pairs} 

In this section, we will characterize several possible configurations of curves in both topological and combinatorial ways. The types of curves we will work with are quasi-homotopic pairs, homotopic pairs, and pants pairs. Unless otherwise stated, we will assume that all curves are essential.

\pit{Disjoint quasi-homotopic pairs and homotopic pairs} A disjoint pair of curves is \emph{quasi-homotopic} if the curves are either homotopic or cobound a pair of pants. (Alternately, the two curves would be homotopic if one puncture were filled in.) %{\color{red}A pair of curves $a,b$ is quasi-homotopic if there exists a curve $a'$ homotopic to $a$ that is disjoint from and quasi-homotopic to $b.$ WILL WE NEED THIS???? CHECK BACK}

\begin{lemma}\label{lemma:disjointQH}
    Let $S=S_0^n$ with $n\geq 6$ and suppose $\varphi\in \Aut \fine(S).$ Then disjoint curves $x,y$ in $S$ are quasi-homotopic if and only if $\varphi(x),\varphi(y)$ are quasi-homotopic.
\end{lemma}

\begin{proof}
    We will show that adjacent curves $x$ and $y$ are quasi-homotopic if and only if one side of the join in $\link(x,y)$ only consists of curves that separate $x$ and $y$. 
    
    Suppose $x$ and $y$ are disjoint quasi-homotopic curves. Then they cobound an annulus or a pair of pants. Since every essential curve in the annulus or pair of pants cobounded by $x$ and $y$ is homotopic to either $x$ or $y,$ it must separate $x$ and $y.$
        
    Conversely, suppose that two adjacent curves $x$ and $y$ are not quasi-homotopic. Then in each component of $S\setminus(x\cup y)$, there exists at least two punctures. We can therefore find a curve in each component of $S\setminus(x\cup y)$ that surrounds two punctures. Thus there exists a curve in each part of the join in $\link(x,y)$ that does not separate $x$ and $y.$
\end{proof}

\begin{proposition}\label{prop:homotopic}
    Let $S=S^n_{0}$ be a surface with $n\geq 6.$ Let $x$ and $y$ be curves in $S$ and let $\varphi\in \Aut\fine(S).$ Then $x$ and $y$ are homotopic if and only $\varphi(x)$ and $\varphi(y)$ are homotopic.
\end{proposition}

\begin{proof}
    Let $x$ and $y$ be disjoint. We claim that $x$ is homotopic to $y$ if and only if they are quasi-homotopic and for all $z\in \link(x,y),$ $z$ is quasi-homotopic to $x$ if and only if $z$ is quasi-homotopic to $y.$ 

    If $x$ is homotopic to $y,$ then any curve disjoint from both $x$ and $y$ is quasi-homotopic to both or neither by the definition of quasi-homotopic.

    Conversely, suppose $x$ is not homotopic to $y.$ Since $n\geq 6,$ without loss of generality, the connected component of $S\setminus x$ not containing $y$ is not a pair of pants. Then the side of $x$ that does not contain $y$ has a curve that is disjoint from and quasi-homotopic to $x$ but not to $y$. 

    \medskip\noindent Suppose now that $x$ and $y$ are not disjoint but are homotopic. By work of Bowden--Hensel--Webb, there is a path between $x$ and $y$ in $\fine(S)$ consisting solely of curves homotopic to $x$ and $y,$ completing the proof \cite{Bowden_Hensel_Webb_2021}.
\end{proof}

A pair of curves $\{x,y\}$ is quasi-homotopic if there exist curves $x'$ homotopic to $x$ and $y'$ homotopic to $y$ such that $x'$ and $y'$ are disjoint and quasi-homotopic. This definition, combined with Proposition~\ref{prop:homotopic}, implies the following lemma.

\begin{lemma}\label{lemma:QH}
    Let $S=S_0^n$ with $n\geq 6$ and suppose $\varphi\in \Aut \fine(S).$ Then curves $x$ and $y$ in $S$ are quasi-homotopic if and only if $\varphi(x)$ and $\varphi(y)$ are quasi-homotopic.
\end{lemma}

\pit{Pants pairs} Curves $a$ and $b$ form a \emph{pants pair} if they are not quasi-homotopic, they are non-crossing, and $a\cap b$ is an interval that is not a single point. The first two conditions in the definition (not quasi-homotopic and non-crossing) are both preserved by automorphisms of $\fine(S)$ by Lemma~\ref{lemma:QH} and Proposition~\ref{prop:sides}. 

The \emph{hull} of a nonempty collection of curves $\Gamma$ is the union of curves in $\Gamma$ and any disks or punctured disks they bound. We will refer to curves contained in the hull of $\Gamma$ as being in the hull of $\Gamma.$ 

We note that the definitions of pants pairs and hulls differ from those of Long--Margalit--Pham--Verberne--Yao to account for the punctures in $S.$

We will use the following lemma of Long--Margalit--Pham--Verberne--Yao characterizing hulls, which applies in our case \cite[Lemma 2.4]{LMPVY}.

\begin{lemma}[Lemma 2.4 of \cite{LMPVY}]\label{lemma:hull}
    Let $S=S^n_{0}$ with $n\geq 5.$ Let $\Gamma$ be a finite nonempty collection of curves in $S$ and let $\varphi\in\Aut \fine(S).$  If $x$ is in the hull of $\Gamma,$ then $\varphi(x)$ is in the hull of $\varphi(\Gamma).$
\end{lemma}

We are now ready to show that pants pairs are preserved by automorphisms of $\fine(S).$

\begin{proposition}\label{prop:pantspairs}
    Let $S=S_0^n$ be a surface with $n\geq 6$ and $\varphi\in \Aut\fine(S).$ Suppose $a,c\in V(\fine(S))$ form a pants pair. Then $\{\varphi(a),\varphi(c)\}$ is also a pants pair. 
\end{proposition}

\begin{proof}
    We will show that two non-disjoint, noncrossing curves $a$ and $c$ form a pants pair if and only if there exists a unique curve $d\neq a,c$ in their hull.
    %    We will show that two non-disjoint, noncrossing curves $a$ and $c$ form a pants pair if and only if: (1) there exists a unique curve $d\neq a,c$ in their hull, and (2) there exists $a'$ homotopic to $a$ such that $a'$ and $c$ are disjoint. 

    If $\{a,c\}$ is a pants pair, then $a\cup c\setminus(\interior(a\cap c))$ (the closure of their symmetric difference) is a curve in the hull of $a$ and $c.$ 

    Conversely, suppose $\{a,c\}$ is not a pants pair. Then, either $|a\cap c|=1$ or the number of connected components of $a\cap c$ is at least 2. If $|a\cap c|=1,$ there are no curves other than $a$ and $c$ in the hull of $a$ and $c.$ Otherwise, if there are at least 2 connected components of $a\cap c,$ we have 2 options: either $a$ and $c$ bound a disk or punctured disk, in which case there are an infinite number of curves in their hull (this can be accomplished by homotoping $a$ into the disk or punctured disk, for example), or $a$ and $c$ do not bound a disk nor a punctured disk. In the latter case, consider a connected component $a'\subset a\setminus c.$ We can close $a'$ into an essential curve by following $c$ in either direction, creating 2 curves other than $a$ and $c$ in the hull of $a$ and $c.$
\end{proof}

\subsection{Bigon pairs}

Our goal in this section is to prove Proposition~\ref{prop:bigonpairs}, which states that bigon pairs are preserved by automorphisms of $\fine(S).$ 

\begin{figure}[h]
\begin{center}
\begin{tikzpicture}
    \node[anchor = south west, inner sep = 0] at (0,0) {\includegraphics[width=0.75\textwidth]{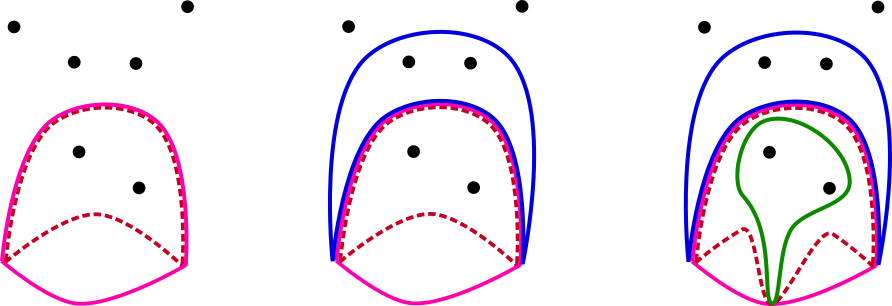}};
    %\draw[help lines] (0,0) grid (12,4);

    %left
    \node at (1.18,0.96){{\color{RRR}$a$}};
    \node at (1.1,0.22){{\color{PPP}$b$}};

    %center
    \node at (5.4,0.96){{\color{RRR}$a$}};
    \node at (5.35,0.22){{\color{PPP}$b$}};
    \node at (5.7,3.7){{\color{BBB}$c$}};

   %right
    \node at (10.35,2.03){{\color{DG}$d$}};
   
    \end{tikzpicture}
\caption{Left: an example of a bigon pair. Center: a curve $c$ that forms a pants pair with both $a$ and $b,$ as required by the proof of Proposition~\ref{prop:bigonpairs}. Right: a curve $d$ as in the proof of Proposition~\ref{prop:bigonpairs}}.\label{fig:bigonpairs}
\end{center}
\end{figure}
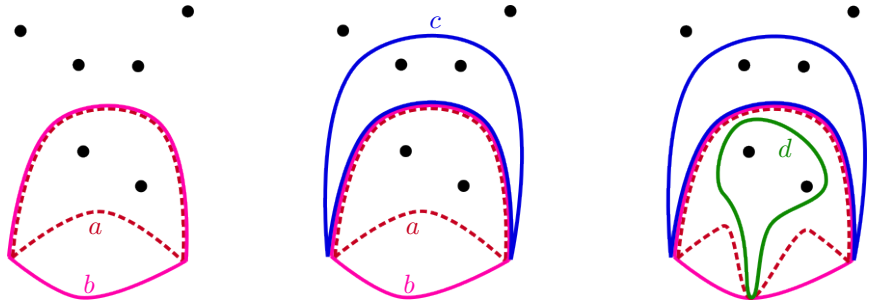

Curves $a$ and $b$ form a \emph{bigon pair} if they are homotopic and $a\cap b$ is a nontrivial interval. An example of a bigon pair can be found on the left of Figure~\ref{fig:bigonpairs}. We call the closure of $a\cup b \setminus(a\cap b)$ the inessential curve \emph{encoded by $a$ and $b$.} 

\begin{proposition}\label{prop:bigonpairs}
    Let $S=S_0^n$ be a boundaryless orientable surface with $n\geq 7.$ Let $a$ and $b$ be vertices of $\fine(S)$ and $\varphi\in\Aut\fine(S).$ Then $\{a,b\}$ is a bigon pair if and only if $\{\varphi(a),\varphi(b)\}$ is a bigon pair.
\end{proposition}

\begin{proof}
    We may assume that $a$ and $b$ are noncrossing and homotopic since these properties are preserved by automorphisms by Propositions~\ref{prop:sides} and \ref{prop:homotopic}.
    
    We will show that $\{a,b\}$ is a bigon pair if and only if there exists an essential curve $c$ such that $\{a,c\}$ and $\{b,c\}$ are both pants pairs and if whenever an essential curve $d$ intersects both $a$ and $b$ but is on the opposite side of $a$ and $b$ as $c$, then $d\cap c \neq \emptyset.$

    Suppose that $a$ and $b$ form a bigon pair. Since at least one side of $a$ and $b$ has at least 4 punctures, we may draw a curve $c$ that intersects $a$ and $b$ at precisely at $a\cap b$ and forms a pants pair with both $a$ and $b.$ (See the center of Figure~\ref{fig:bigonpairs} for an example.) Without loss of generality, $c$ is on the same side of $a$ as $b.$ Then for every curve $d$ on the opposite side of $a,$ in order to intersect both $a$ and $b$ and not cross $a,$ $d$ must intersect $a\cap b,$ and therefore also intersect $c.$

    Suppose now that $a$ and $b$ do not form a bigon pair. Since $a$ and $b$ are noncrossing and homotopic, $|a\cap b|=1$ or there are multiple connected components of $a\cap b.$ In the former case, there is no curve that forms a pants pair with both $a$ and $b$ since such a curve must intersect both $a$ and $b$ in a nontrivial interval while also not crossing them. In the latter case, let $c$ be a pants pair with both $a$ and $b$ (if one exists; if one does not, we are done) and let $d\in \link(a,b)$ be on the opposite side of $a$ and $b$ as $c.$ Then we can isotope $d$ to intersect a connected component of $a\cap b$ that does not intersect $c,$ as desired. (See the right of Figure~\ref{fig:bigonpairs} for an example.)
\end{proof}   

The key intuitive idea is that a bigon pair encodes an inessential curve, and the above proposition states that a bigon pair is sent to a bigon pair. We will use this fact to extend an automorphism of $\fine(S)$ to an automorphism of $\mathcal{E}\fine(S)$ by defining how the automorphism acts on inessential curves. It remains to show that two bigon pairs encoding the same inessential curve are sent to bigon pairs encoding the same inessential curve; that is the point of the next section and of sharing pairs.

\subsection{Sharing pairs} In this section, our main goal is to prove that sharing pairs are preserved by automorphisms; this is Proposition~\ref{prop:sharingpairs}.

A \emph{sharing pair} is a pair of bigon pairs that encode the same inessential curve. An example of a sharing pair is in the left of Figure~\ref{fig:nonQHstandardsharingpairsexp}. To ensure that our encodings of inessential curves and their images are well-defined, we must ascertain that sharing pairs are taken to sharing pairs by automorphisms of $\fine(S).$

\begin{proposition}\label{prop:sharingpairs}
    Let $S=S_0^n$ be a boundaryless orientable surface with $n\geq 7.$ Suppose $\{a_1,a_2\}$ and $\{b_1,b_2\}$ are bigon pairs that form a sharing pair. Let $\varphi \in \Aut \fine(S).$ Then $\{\varphi(a_1),\varphi(a_2)\}$ and $\{\varphi(b_1),\varphi(b_2)\}$ also form a sharing pair.
\end{proposition}

To prove this, we will go through several steps. A \emph{standard sharing pair} $\{a,b\}$ and $\{c,d\}$ encoding inessential curve $e$ is a sharing pair such that the closures of the arcs $a\setminus e$ and $c\setminus e$ are disjoint, including at their endpoints. 

We will first show that non-quasi-homotopic standard sharing pairs are preserved by automorphisms of $\fine(S),$ then that standard sharing pairs are preserved, and finally that all sharing pairs are preserved.

\begin{lemma}\label{lemma:nonQHstandardsharingpairs}
    Let $S=S_0^n$ be a boundaryless orientable surface with $n\geq 7.$ Suppose $\{a_1,a_2\}$ and $\{b_1,b_2\}$ are non-quasi-homotopic bigon pairs that form a standard sharing pair. Let $\varphi \in \Aut \fine(S).$ Then $\{\varphi(a_1),\varphi(a_2)\}$ and $\{\varphi(b_1),\varphi(b_2)\}$ also form a non-quasi-homotopic standard sharing pair.
\end{lemma}

\begin{proof}
    We will show that $\{a_1,a_2\}$ and $\{b_1,b_2\}$ are non-quasi-homotopic bigon pairs that form a standard sharing pair if and only if the following conditions hold, up to reordering of the $a_i$s and $b_i$s: 
    \begin{enumerate}
        \item $\{a_1,b_1\}$ is a pants pair,
        \item $\{a_2,b_2\}$ is a pants pair, 
        \item $a_1$ and $b_2$ are disjoint, 
        \item $a_2$ and $b_1$ are crossing, and 
        \item there does not exist an essential curve $c$ such that $c$ intersects exactly one of $\{a_1,a_2,b_1,b_2\}.$
    \end{enumerate}

    Suppose that $\{a_1,a_2\}$ and $\{b_1,b_2\}$ are non-quasi-homotopic bigon pairs that form a standard sharing pair. Then, up to homeomorphism (and adjusting for the number of punctures in the surface), they are arranged as in the left (or, equivalently, middle) of Figure~\ref{fig:nonQHstandardsharingpairsexp}. As seen in Figure~\ref{fig:nonQHstandardsharingpairsexp}, requirements (1)-(4) are satisfied by observation. (5) is satisfied since there is no point in $a_1\cup a_2\cup b_1\cup b_2$ that belongs to precisely one of the curves.

    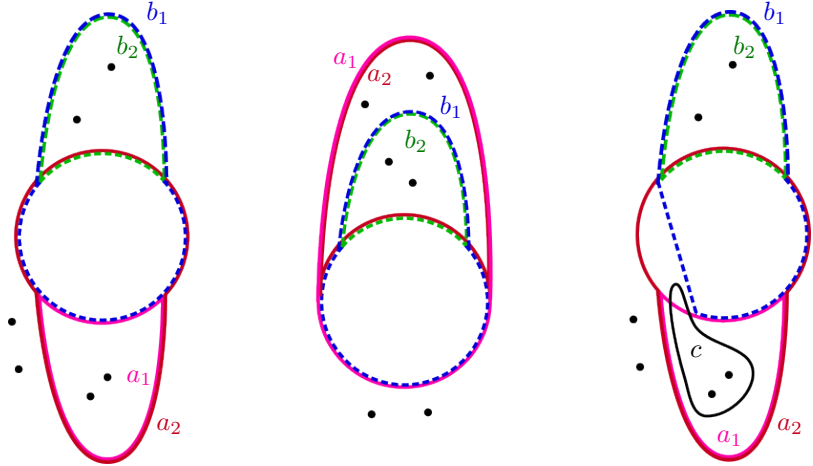
\begin{figure}[h]
\begin{center}
\begin{tikzpicture}
    \node[anchor = south west, inner sep = 0] at (0,0) {\includegraphics[width=0.7\textwidth]{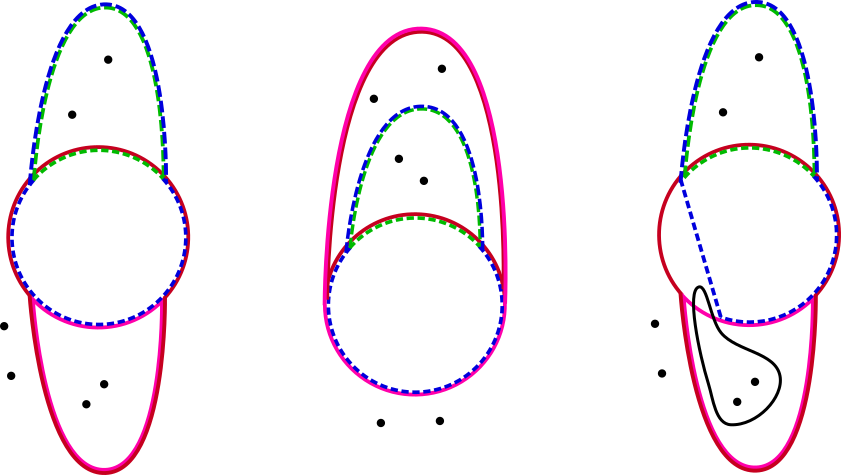}};
    %\draw[help lines] (0,0) grid (12,7);

    %left
    \node at (1.6,5.56){{\color{GGG}$b_2$}};
    \node at (2,6){{\color{BBB}$b_1$}};
    \node at (1.75,1.15){{\color{PPP}$a_1$}};
    \node at (2.15,0.5){{\color{RRR}$a_2$}};

    %middle
    \node at (5.43,4.3){{\color{GGG}$b_2$}};
    \node at (5.9,4.75){{\color{BBB}$b_1$}};
    \node at (4.5,5.37){{\color{PPP}$a_1$}};
    \node at (4.95,5.15){{\color{RRR}$a_2$}};

    %right
    \node at (9.8,5.56){{\color{GGG}$b_2$}};
    \node at (10.2,6){{\color{BBB}$b_1$}};
    \node at (9.6,0.35){{\color{PPP}$a_1$}};
    \node at (10.4,0.5){{\color{RRR}$a_2$}};
    \node at (9.15,1.5){$c$};
    
    \end{tikzpicture}
\caption{Left and center: Examples of two equivalent ways to picture sharing pairs. Right: Construction of a curve $c$ that intersects exactly one of $a_1,\ a_2,\ b_1,$ and $b_2$ if the attaching arc of $a$ does not attach in $e_a\cap e_b$.}\label{fig:nonQHstandardsharingpairsexp}
\end{center}
\end{figure}

    Conversely, suppose that $\{a_1,a_2\}$ and $\{b_1,b_2\}$ are not quasi-homotopic bigon pairs that do not form a sharing pair. In other words, the inessential curves $e_a$ and $e_b$ defined by $\{a_1,a_2\}$ and $\{b_1,b_2\}$, respectively, are not the same. We can view $a_1\cup a_2$ as attaching an arc (comprised of $(a_1\cup a_2)\setminus e_a$) to $e_a$; we call this arc the \emph{attaching arc} of $a.$ Similarly, we have an attaching arc of $b.$

    %{\color{blue}
        Due to the disjointness condition (3), we have that the attaching arcs of $a$ and $b$ are disjoint. Due to the pants conditions (1) and (2), $a_1\cap b_1$ must consist of a single connected component and $a_2\cap b_2$ must consist of a single connected component. Moreover, $a_1\cap e_a \subset e_b$ to avoid having a curve parallel to $a_1$ intersect exclusively $a_1$, as in the right of Figure~\ref{fig:nonQHstandardsharingpairsexp}. Similarly, $b_2\cap e_b\subset a_2.$ We can then sort the punctures into three categories: on the side of $b_2$ not containing $a_1$, on the side of $a_1$ not containing $b_2,$ and all the others (the ``outsiders''). If $e_a\neq e_b$, it means $a_2$ does not completely coincide with $b_1$ outside the attaching arc of $a$ (and vice versa for the attaching arc of $b$). 

        We now have two options: either $e_a$ is in one connected component of $S\setminus e_b$ (without loss of generality, it is in the disk) or in both connected components of $S\setminus e_b.$ Either way, draw a curve $c$ disjoint from $a_2$ surrounding the outsider punctures and (if necessary) homotope it to intersect $b_1\cap e_b,$ breaking condition (5).
\end{proof}

Now that we know that non-quasi-homotopic standard sharing pairs are preserved by automorphisms, we want to prove the analogous result for all standard sharing pairs.

\begin{lemma}\label{lemma:standardsharingpairs}
    Let $S=S_0^n$ be a boundaryless orientable surface with $n\geq 7.$ Suppose $A=\{a_1,a_2\}$ and $B=\{b_1,b_2\}$ are bigon pairs that form a standard sharing pair. Let $\varphi \in \Aut \fine(S).$ Then $\{\varphi(a_1),\varphi(a_2)\}$ and $\{\varphi(b_1),\varphi(b_2)\}$ also form a standard sharing pair.
\end{lemma}

\begin{proof}
    The result is already proven if $a_i$ and $b_j$ are not quasi-homotopic in Lemma~\ref{lemma:nonQHstandardsharingpairs}, so we suppose they are. The region between $a_1$ and $b_2$, pictured on the outside (so to speak) of the curves in both sides of Figure~\ref{fig:standardsharingpairs}, contains 0 or 1 punctures, so either: (1) one of $a_1$ and $b_2$ bounds at least 4 punctures on the side opposite $b_2$ and $a_1,$ respectively, or (2) both $a_1$ and $b_2$ bound exactly 3 punctures on the side opposite $b_2$ and $a_1,$ respectively. In case (1) (pictured on the left of Figure~\ref{fig:standardsharingpairs}), we can find a bigon pair $C=\{c_1,c_2\}$ that forms a non-quasi-homotopic standard sharing pair with both $\{a_1,a_2\}$ and $\{b_1,b_2\}$ and is in the connected component of $S\setminus(a_1\cup a_2\cup b_1\cup b_2)$ containing at least 4 punctures. In case (2) (pictured on the right of Figure~\ref{fig:standardsharingpairs}), we can find two bigon pairs $C=\{c_1,c_2\}$ and $D=\{d_1,d_2\}$ such that the following form a non-quasi-homotopic standard sharing pairs: $\{A,C\}$, $\{C,D\},$ and $\{D,B\}.$ This can be done by having the attaching arc of the curves in $C$ be contained in the thrice punctured disk bounded by the $B$ curves and having the attaching arc of the curves in $D$ be contained in the thrice punctured disk bounded by the curves in $A.$
    \begin{figure}[h]
\begin{center}
\begin{tikzpicture}
    \node[anchor = south west, inner sep = 0] at (0,0) {\includegraphics[width=0.5\textwidth]{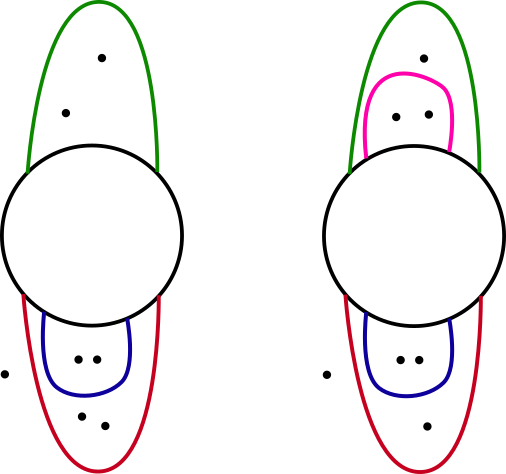}};
    %\draw[help lines] (0,0) grid (8,8);

    %left
    \node at (0.6,6.5){{\color{DG}$A$}};
    \node at (1.3,1.46){{\color{DB}$C$}};
    \node at (1.5,0.3){{\color{RRR}$B$}};

    %right
    \node at (5.45,6.5){{\color{DG}$A$}};
    \node at (6.1,1.46){{\color{DB}$C$}};
    \node at (6.45,0.3){{\color{RRR}$B$}};
    \node at (6.2,5.7){{\color{PPP}$D$}};

    \end{tikzpicture}
\caption{Left: case (1) in the proof of Lemma~\ref{lemma:standardsharingpairs}. Right: case (2) in the proof of Lemma~\ref{lemma:standardsharingpairs}.}\label{fig:standardsharingpairs}
\end{center}
\end{figure}
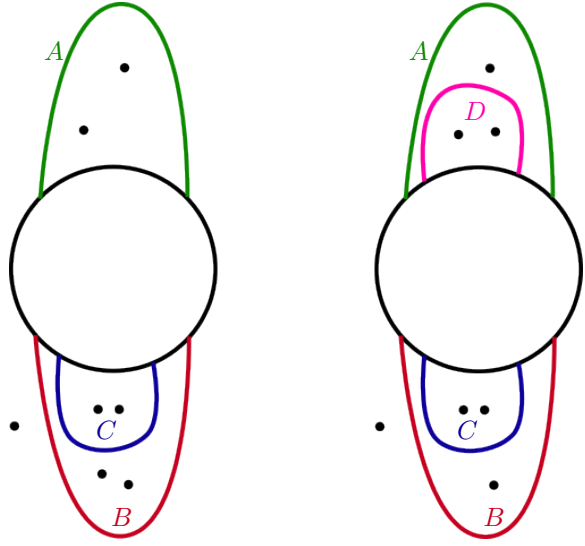
\end{proof}

With this in mind, we are almost ready to prove Proposition~\ref{prop:sharingpairs}. The proof idea will be the same as that of the standard sharing pairs, as we will want to show that between any two sharing pairs, there is a path of standard sharing pairs. However, this is not immediate, and we will pass through an arc graph to prove this result.

\pit{Connectivity of an arc graph variant} Let $D_n$ be a surface with 1 boundary and $n\geq7$ punctures (a punctured disk). Define $\weirdarc(D_n)$ to be a graph whose vertices are isotopy classes of properly embedded arcs whose endpoints are in $\partial D_n$ and that bound at least 2 punctures and whose edges connect arcs that admit disjoint representatives. We allow the endpoints of the arcs to move in the boundary as part of isotopies. We say an arc $\gamma$ \emph{surrounds $k$ punctures} if the complementary component of $\gamma$ with the fewest punctures contains $k$ punctures.

Define the \emph{mapping class group} of a surface $S$, denoted $\MCG(S)$, to be the group of isotopy classes of homeomorphisms of $S$ that fix the boundary pointwise (equivalently, $\pi_0(\Homeo(S,\partial S)).$) 

\begin{proposition}\label{prop:weirdarcgraphconnected}
    Let $D_n$ be an orientable surface with 1 boundary and $n\geq 7$ punctures. Then $\weirdarc(D_n)$ is connected.
\end{proposition}

\begin{proof}
    To prove connectivity, we will use Putman's trick \cite{putman}. First, we note that $D_n$ is a punctured disk, so its mapping class group is the braid group $B_n.$ Let $Q=\{\sigma_1,\ldots,\sigma_{n-1}\}$ be the generating set for $B_n$ consisting of $n-1$ half-twists, as shown in the left of Figure~\ref{fig:weirdarcgraph}. $\MCG(D_n)$ acts on $\weirdarc(D^n)$ in the natural manner. For the remainder of the proof, we assume that all endpoints of arcs are in the boundary component of $D_n.$ Let $v_0$ be an arc surrounding the first 2 punctures, in the sense of the right of Figure~\ref{fig:weirdarcgraph}.

        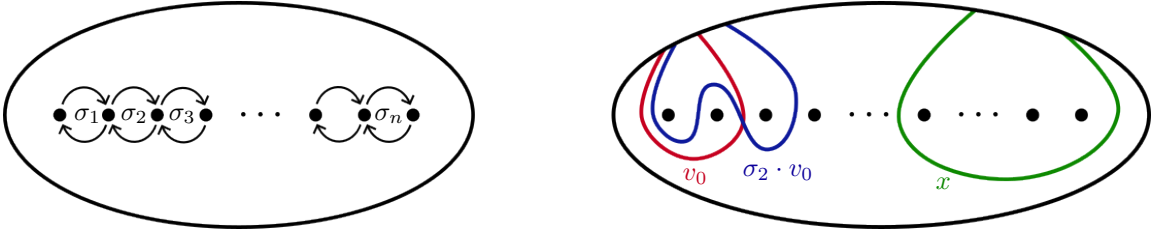
\begin{figure}[h]
\begin{center}
\begin{tikzpicture}
    \node[anchor = south west, inner sep = 0] at (0,0) {\includegraphics[width=\textwidth]{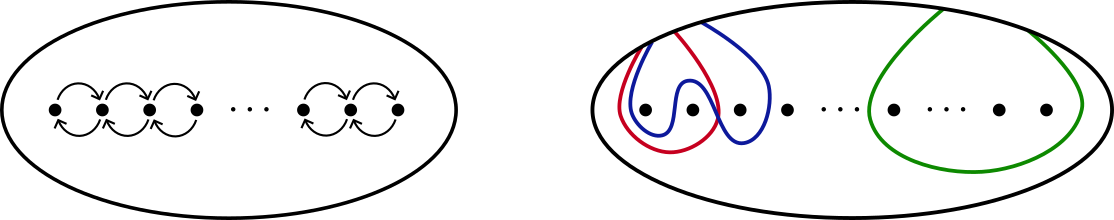}};
    %\draw[help lines] (0,0) grid (16,4);
    %left
    \node at (1.12,1.5){$\sigma_1$};
    \node at (1.74,1.5){$\sigma_2$};
    \node at (2.38,1.5){$\sigma_3$};
    \node at (5.13,1.5){$\sigma_{n}$};
    
    %right
    \node at (9.2,0.7){{\color{RRR}$v_0$}};
    \node at (10.3,0.75){{\color{DB}$\sigma_2\cdot v_0$}};
    \node at (12.5,0.6){{\color{DG}$x$}};

    \end{tikzpicture}
\caption{Left: nested endpoints of attaching arcs. Right: unnested endpoints of attaching arcs.}\label{fig:weirdarcgraph}
\end{center}
\end{figure}

    We must now check that the two conditions for Putman's trick hold. Let $v\in V(\weirdarc(D_n))$ be an arc that surrounds $k \leq \frac{n}{2}$ punctures. We wish to show that there is an element of the orbit of $v$ under the mapping class group action that is in the connected component of $v_0$ in $\weirdarc(D_n)$. The mapping class group orbit of $v$ consists of all arcs that surround $k$ punctures (by the classification of surfaces). Such an arc is pictured as $x$ in the right of Figure~\ref{fig:weirdarcgraph}. In fact, we can guarantee that $x$ is disjoint from $v_0$ since $2\leq \frac{n}{2}$: there are $n-k$ punctures outside a model arc that surrounds $k$ punctures, and $v_0$ surrounds 2 of them.

    Next, we need to show that for all $\sigma_i \in Q^\pm,$ there is a path from $v_0$ to $q\cdot v_0.$ There are two cases here. If $i\neq 2$ (so $\sigma_i$ is a half twist on the punctures $v_0$ surrounds or does not surround, but not a combination of both), then $\sigma_i\in \text{Stab}(v_0)$, so the path from $v_0$ to $\sigma_i\cdot v_0$ is just $v_0.$ Otherwise, $\sigma_2\cdot v_0$ is distance 2 from $v_0,$ as shown in the right hand side of Figure~\ref{fig:weirdarcgraph}. 
    
  Thus $\weirdarc(D_n)$ is connected by the Putman trick. 
\end{proof}

\pit{Back to the sharing pairs proof} With the above in mind, we are ready to complete the proof that sharing pairs are preserved by automorphisms of the fine curve graph.

\begin{proof}[Proof of Proposition~\ref{prop:sharingpairs}]
    If $A=\{a_1,a_2\}$ and $B=\{b_1,b_2\}$ form a standard sharing pair, we are done by Lemma~\ref{lemma:standardsharingpairs}. Otherwise, we will show that there is a sequence $X^i=\{x^i_1,x^i_2\}$ of standard sharing pairs such that $X^0=A,$ $X^p=B,$ and $X^i$ and $X^{i+1}$ form a standard sharing pair for all $0\leq i\leq p-1.$ Let $e$ be the inessential curve encoded by $A$ and $B.$ 
    
    It is enough to show that there is a sequence $\alpha_0,\ldots,\alpha_p$ of properly embedded arcs in $S$ with the open disk bounded by $e$ removed such that $\alpha_i$ is disjoint from $\alpha_{i+1}$ for $0\leq i\leq p-1$, $\alpha_0$ is the attaching arc for $A$, $\alpha_p$ is the attaching arc for $B,$ and the endpoints of all arcs are in $e.$

    First, consider the graph $\weirdarc(D_n),$ where $D_n$ is $S$ with the open disk bounded by $e$ removed. There are representatives of all isotopy classes of arcs such that all arcs are in pairwise minimal position. Consider the representatives for the attaching arc of $A$ (call this $\Tilde{a}$) and the attaching arc of $B$ (call this arc $\Tilde{b}$). Since $\weirdarc(D_n)$ is connected, there is a sequence of arcs that connects $\Tilde{a}$ to $\Tilde{b}.$ 

    It remains to show that two arcs in the same isotopy class in $D_n$  whose endpoints are in $\partial D_n$ can be connected by a sequence of arcs that are all pairwise homotopic. This proof is identical to the proof of Proposition 3.1 of Long--Margalit--Pham--Verberne--Yao \cite{LMPVY}. (We note that the reason their proof does not work for surfaces with punctures is due to arcs whose endpoints are at punctures because an arc can swirl around a puncture infinitely many times if it ends there.) 

    We use the above paths of arcs to construct a sequence of standard sharing pairs between $A$ and $B,$ as desired.
\end{proof}

\section{Proof of Theorem~\ref{maintheorem}}

In this section, we complete the proof of the main theorem: the natural map $\Homeo(S)\to \Aut \fine(S)$ is an isomorphism.

\begin{proof}[Proof of Theorem~\ref{maintheorem}]
    Our proof is informed by the following commutative diagram:

\begin{center}
    \begin{tikzcd}
 \Aut \fine(S) \arrow{r}{\Psi} & \Aut \mathcal{E}\fine(S) \arrow{r}{\nu^{-1}} &\Homeo(S) \arrow[bend left]{ll}{\Phi}
 \end{tikzcd}
 \end{center}

 We recall that $\Phi$ is the natural map from the theorem statement and $\nu$ is the map from Theorem~\ref{theorem:efcg}. In Step 1 below, we define the map $\Psi.$ In Step 2 below, we show that $\nu^{-1}\circ \Psi\circ \Phi=\id,$ completing the proof that $\Phi$ is an isomorphism.

 \pit{Step 1: defining $\Psi$} Let $\varphi\in \Aut\fine(S)$. Since $\fine(S)$ is a subgraph of $\mathcal{E}\fine(S),$ we wish to extend $\varphi$ to an automorphism $\phi$ of $\mathcal{E}\fine(S)$ so that $\Psi(\varphi)=\phi.$ To do so, we first define $\phi(x)=\varphi(x)$ for all essential curves $x.$ 
 
 Let $e$ be an inessential curve; we define $\phi(e)$ as follows. Let $\{a,b\}$ be a bigon pair that encodes $e.$ Then $\{\varphi(a),\varphi(b)\}$ encodes an inessential curve $e'.$ Define $\phi(e)=e'.$ This is well-defined since all bigon pairs that encode $e$ have images that encode $e';$ this is because they form sharing pairs, which are preserved by Proposition~\ref{prop:sharingpairs}. 

 To conclude our proof that $\phi$ is an automorphism of $\mathcal{E}\fine(S),$ we must show that $\phi$ preserves adjacency (edges); non-adjacency will follow by considering the inverse map. Our proof is identical to that of Long--Margalit--Pham--Verberne--Yao, but we reiterate it for completeness. There are three types of edges based on what types of curves they connect, which we treat case-by-case below.

\begin{adjustwidth}{1cm}{}
\pit{Case 1: Essential-Essential} These edges are preserved by the definition of $\phi,$ which comes from $\varphi.$

\pit{Case 2: Essential-Inessential} In this case, the complementary components of the essential curve both have at least 2 punctures, so we can find a bigon pair that encodes the inessential curve that is also disjoint from the essential curve.

\pit{Case 3: Inessential-Inessential} Two inessential curves $e$ and $f$ are disjoint if and only if for all bigon pairs $\{a,b\}$ that encode $e$, we can construct a bigon pair $\{c,d\}$ that encodes $f$ such that both $c$ and $d$ are disjoint from $a.$ (This is up to relabeling of $e$ and $f$.)

Suppose $e$ and $f$ are disjoint. If they are not nested (so neither is contained in the disk bounded by the other), we use the fact that the component of $S\setminus(a\cup b)$ that contains $f$ has at least 2 punctures, so we can construct a bigon pair $\{c,d\}$ in this component, and therefore disjoint from both $a$ and $b.$ If $e$ and $f$ are nested, we must take $e$ to be the outer curve. To construct a bigon pair $\{c,d\}$ encoding $f,$ route the curves through the part of $e$ that is part of $e\cap b$, and then follow the remainder of $b$ to close up the curve.

Conversely, suppose $e$ and $f$ are not disjoint, and let $x\in e\cap f.$ Define a bigon pair $\{a,b\}$ such that $x\in a\cap b$ (so it is a ``corner'' of the bigon). Since every bigon pair that encodes $f$ intersects $x,$ it must also intersect both $a$ and $b.$
\end{adjustwidth}

\pit{Step 2: $\nu^{-1}\circ \Psi\circ \Phi=\id$} We first note that $\Phi$ is injective; the proof is the same as that of Long--Margalit--Pham--Verberne--Yao \cite[Lemma 5.1]{LMPVY}. 

Let $\psi\in \Homeo(S)$ and define $\varphi_\psi = \Phi(\psi)$ to be the automorphism of $\fine(S)$ that acts the same way on curves in $S$ as $\psi.$ Let $\phi_\psi = \Psi(\varphi_\psi)$ be the extension to the automorphisms of $\mathcal{E}\fine(S).$ Since there is only one homeomorphism that acts in the same manner on essential curves as $\phi_\psi$ (because $\Phi$ is injective), we have that $\nu^{-1}(\phi_\psi)=\psi.$ We therefore have that
\[\nu^{-1}\circ \Psi\circ \Phi(\psi)=\nu^{-1}\circ\Psi(\varphi_\psi) = \nu^{-1}(\phi_\psi) = \psi,\]
as desired.
\end{proof}

\printbibliography

\end{document}